\documentclass[12pt, reqno]{amsart}

\author[M.~Balcerzak]{Marek Balcerzak}
\address{Institute of Mathematics, Lodz University of Technology, ul. W\'{o}lcza\'{n}ska 215, 93-005 Lodz, Poland} % \L{}\'{o}d\'{z}
\email{marek.balcerzak@p.lodz.pl}

\author{Paolo Leonetti}
\address{Institute of Analysis and Number Theory, Graz University of Technology,  Kopernikusgasse 24/II, 8010 Graz, Austria}
\email{leonetti.paolo@gmail.com}

\thanks{P.L. is supported by the Austrian Science Fund (FWF), project F5512-N26.}


\title{A Tauberian theorem for ideal statistical convergence}

\usepackage[T1]{fontenc}
\usepackage{amsmath}
\usepackage{amssymb}
\usepackage{amsthm}
\usepackage[left=3.5cm, right=3.5cm, paperheight=11.8in]{geometry}
\usepackage{hyperref}
\usepackage{fancyhdr}
\usepackage{enumitem}
\usepackage{comment}
\usepackage{nicefrac}
\usepackage{bm}
\usepackage{mathrsfs}
\usepackage{graphicx}
\usepackage[utf8]{inputenc}
\usepackage{cancel}
\usepackage{mathtools}

\usepackage{tikz}
\usetikzlibrary{decorations.pathreplacing,matrix,arrows,positioning,automata,shapes,shadows,calc,fadings,decorations,snakes,through,intersections}

\AtBeginDocument{%
   \def\MR#1{}
}

\newtheorem{thm}{Theorem}[section]
\newtheorem{cor}[thm]{Corollary}%[section]
\newtheorem{lem}[thm]{Lemma}

\theoremstyle{definition} 
\newtheorem{defi}[thm]{Definition}%[section]
\let\olddefi\defi
\renewcommand{\defi}{\olddefi\normalfont}
\newtheorem{question}{Question}%[section]
\let\oldquestion\question
\renewcommand{\question}{\oldquestion\normalfont}
\newtheorem{example}[thm]{Example}
\let\oldexample\example
\renewcommand{\example}{\oldexample\normalfont}
\newtheorem{rmk}[thm]{Remark}
\let\oldrmk\rmk
\renewcommand{\rmk}{\oldrmk\normalfont}
\newtheorem{claim}{\textsc{Claim}}

\hypersetup{
    pdftitle={A Tauberian theorems for ideal statistical convergence},
    pdfauthor={Marek Balcerzak and Paolo Leonetti},
    pdfmenubar=false,
    pdffitwindow=true,
    pdfstartview=FitH,
    colorlinks=true,
    linkcolor=blue,
    citecolor=green,
    urlcolor=cyan
}

\providecommand{\MR}[1]{}

\providecommand{\MR}{\relax\ifhmode\unskip\space\fi MR }

\providecommand{\href}[2]{#2}

\begin{document}

\maketitle
\thispagestyle{empty}

\begin{abstract}
Given an ideal $\mathcal{I}$ on the positive integers, a real sequence $(x_n)$ is said to be $\mathcal{I}$-statistically convergent to $\ell$ provided that 
$$
\textstyle \left\{n \in \mathbf{N}: \frac{1}{n}|\{k \le n: x_k \notin U\}| \ge \varepsilon\right\} \in \mathcal{I}
$$
for all neighborhoods $U$ of $\ell$ and all $\varepsilon>0$. 
First, we show that $\mathcal{I}$-statistical convergence coincides with $\mathcal{J}$-convergence, for some unique ideal $\mathcal{J}=\mathcal{J}(\mathcal{I})$. In addition, $\mathcal{J}$ is Borel [analytic, coanalytic, respectively] whenever $\mathcal{I}$ is Borel [analytic, coanalytic, resp.]. 

Then we prove, among others, that if $\mathcal{I}$ is the summable ideal $\{A\subseteq \mathbf{N}: \sum_{a \in A}1/a<\infty\}$ or the density zero ideal $\{A\subseteq \mathbf{N}: \lim_{n\to \infty} \frac{1}{n}|A\cap [1,n]|=0\}$ then $\mathcal{I}$-statistical convergence coincides with 
statistical convergence. 
This can be seen as a Tauberian theorem which extends a classical theorem of Fridy. 
Lastly, we show that this is never the case if $\mathcal{I}$ is maximal.
\end{abstract}

%%%%%%%%%%%%%%%%%%%%%%%%%%%%%%%%%%%%%%%%%%%%%%%%%%%%%%%%%%%%%%%%%%%%%%%%%%%%

\section{Introduction}\label{sec:intro}

Let $\mathcal{I}\subseteq \mathcal{P}(\mathbf{N})$ be an ideal, that is, a collection of subsets of the positive integers $\mathbf{N}$ closed under taking finite unions and subsets. 
It is also assumed that $\mathcal{I}$ contains the collection $\mathrm{Fin}$ of finite subsets of $\mathbf{N}$ and that, unless otherwise stated, $\mathcal{I}$ is proper, that is, it is different from $\mathcal{P}(\mathbf{N})$. 
Among the most important ideals we can find the family of asymptotic density zero sets 
$$
\mathcal{Z}:=\left\{A\subseteq \mathbf{N}: \lim_{n\to \infty}\frac{|A\cap [1,n]|}{n}=0\right\}
$$
and the summable ideal
$$
\mathcal{I}_{1/n}:=\left\{S\subseteq \mathbf{N}: \sum_{n \in S}\frac{1}{n}<\infty\right\}.
$$

Let $X$ be a Hausdorff topological space. Given an ideal $\mathcal{I}$, a sequence $(x_n)$ taking values in $X$ is said to be $\mathcal{I}$-convergent to $\ell \in X$, in short $x_n \to_{\mathcal{I}} \ell$, if 
$$
\{n \in \mathbf{N}: x_n \notin U\} \in \mathcal{I}
$$
for all neighborhoods $U$ of $\ell$. In the literature, $\mathcal{Z}$-convergence is usually called \emph{statistical convergence}, see \cite{MR3452220, MR2463821} and references therein. We recall that, if $\mathcal{I}\neq \mathrm{Fin}$ and $X$ has at least two distinct points, then $\mathcal{I}$-convergence does \emph{not} correspond to ordinary convergence with respect to any topology on the same base set, see \cite[Example 2.2]{LMxyz} and \cite[Proposition 4.2]{MR1844385}. 
In particular, the notion of ideal convergence is a 
%``strict generalization'' 
``proper extension'' 
of classical convergence.

Recently, Das and Savas introduced in \cite{MR2776143} the notion of $\mathcal{I}$\emph{-statistical convergence}: a sequence $(x_n)$ taking values in $X$ is said to be $\mathcal{I}$-convergent to $\ell \in X$ if 
\begin{equation}\label{eq:Istatconv}
\left\{n \in \mathbf{N}: \frac{|\{k \in [1,n]: x_k \notin U\}|}{n} \ge \varepsilon\right\} \in \mathcal{I}
\end{equation}
for all neighborhoods $U$ of $\ell$ and all $\varepsilon>0$ (note that the original definition has been given in the context of normed spaces). The aim of this article is threefold.

First, the authors of \cite{MR2776143} remark that $\mathrm{Fin}$-statistical convergence corresponds to statistical convergence, cf. also \cite[Remark 1]{MR2803700}. Hence, one may wonder whether $\mathcal{I}$-statistical convergence corresponds to $\mathcal{J}$-convergence, for some ideal $\mathcal{J}=\mathcal{J}(\mathcal{I})$. We give a positive answer, in a slightly more general context, see Theorem \ref{thm:marek}.

Second, the same authors claim in \cite[Remark 2]{MR2776143} that there exists a sequence $(x_n)$ which is $\mathcal{Z}$-statistically convergent but not statistically convergent. However, it turns out that their claim is false. Indeed, we show that $\mathcal{Z}$-statistical convergence and statistical convergence coincide, see Theorem \ref{thm:character} and Corollary \ref{cor:idealstat}. As we will explain in the next Section, this is a Tauberian theorem which extends a classical result of Fridy \cite[Theorem 3]{MR816582}. Related results have been extensively studied in the literature, see e.g. \cite{MR2324333, MR3906365, 
MR1659877, MR1653457, MR1989684, 
MR938459, MR1002541, MR1941785, MR2009641, MR2079327}.

Lastly, on the opposite direction, we prove that $\mathcal{I}$-statistical convergence \emph{never} coincides with statistical convergence whenever $\mathcal{I}$ is maximal, see Theorem \ref{thm:maximal}.% and Corollary \ref{cor:maximal}.}

%%%%%%%%%%%%%%%%%%%%%%%%%%%%%%%%%%%%%%%%%%%%%%%%%%%%%%%%%%%%%%%%%%%%%%%%%%%%%%%%

\section{Main results}

An ideal $\mathcal{I}$ is said to be a P-ideal if it is $\sigma$-directed modulo finite sets, i.e., for each sequence $(A_n)$ in $\mathcal{I}$ there exists $A \in \mathcal{I}$ such that $A_n\setminus A$ is finite for all $n$. 
%Also, $\mathcal{I}$ is tall if each infinite $A\subseteq \mathbf{N}$ has an infinite subset in $\mathcal{I}$. 
By identifying sets of integers with their characteristic functions, we equip $\mathcal{P}(\mathbf{N})$ with the Cantor-space topology and therefore we can assign the topological complexity to the ideals on $\mathbf{N}$. In particular, we can speak about Borel ideals, analytic ideals, meager ideals, etc. It is a folklore result that the the ideals with lowest topological complexity are $F_{\sigma}$-ideals. 
We refer to \cite{MR2777744} for a recent survey on ideals and filters.

A map $\varphi: \mathcal{P}(\mathbf{N}) \to [0,\infty]$ is a submeasure provided that for all $A,B\subseteq \mathbf{N}$: (i) $\varphi(\emptyset)=0$, (ii) $\varphi(A) \le \varphi(B)$ if $A\subseteq B$, (iii) $\varphi(A\cup B) \le \varphi(A)+\varphi(B)$, and (iv) $\varphi(\{n\})<\infty$ for all $n$. In addition, a submeasure $\varphi$ is lower semicontinuous if: (v) $\varphi(A)=\lim_{n\to \infty}\varphi(A\cap [1,n])$ for all $A$. 
By a classical result of Solecki, a (not necessarily proper) ideal $\mathcal{I}$ is an analytic P-ideal if and only if there exists a lower semicontinuous submeasure $\varphi$ such that $\mathcal{I}$ coincides with the exhaustive ideal $\mathrm{Exh}(\varphi)$ generated by $\varphi$, that is, 
$$
\mathcal{I}=\mathrm{Exh}(\varphi):=\{A\subseteq \mathbf{N}: \lim_{n\to \infty} \varphi(A\setminus [1,n]) =0\} 
$$ 
and $\varphi(\mathbf{N})<\infty$, cf e.g. \cite[Theorem 1.2.5]{MR1711328}. 

\begin{defi}\label{def:smoothsubsmeasuresequence}
A sequence of submeasures $\mu=(\mu_n)$ is said to be \emph{smooth} provided that: 
\begin{enumerate}[label={\rm (\textsc{s}\arabic{*})}]
\item \label{item:s1} for all $n \in \mathbf{N}$, $\mu_n$ is supported on a nonempty set $I_n$;
\item \label{item:s2} $\lim_{n\to \infty} \mu_n(\{k\})=0$ for all $k \in \mathbf{N}$;
%\textcolor{red}{(ii) $\mu_n(\{k\})\not\to_{\mathcal{I}}0$ for all $k \in \mathbf{N}$, }
%and (iii) $\limsup_{n}\mu_n(I_n)>0$. 
\item \label{item:s3} $\limsup_{n\to \infty} \mu_n(\mathbf{N})>0$. 
%$\mu_n(I_n)\not\to_{\mathcal{I}}0$.
\end{enumerate}

\noindent In this regard, let $\mathcal{Z}_\mu$ be the ideal defined by
$$
\mathcal{Z}_\mu:=\left\{A\subseteq \mathbf{N}: \limsup_{n\to \infty}\mu_n(A \cap I_n)=0\right\}.
$$
\end{defi}

%Note that, 
%%as it written, Definition \ref{def:smoothsubsmeasuresequence} is too general. Indeed, 
%if $\mathcal{I}$ and $\mathcal{J}$ are ideals on $\mathbf{N}$, then there exists an $\mathcal{I}$-smooth sequence of submeasures $\mu$ such that $\mathcal{J}=\mathcal{Z}_\mu$: indeed, it is sufficient to set $\mu_n(A)=\bm{1}_{\mathcal{P}(\mathbf{N})\setminus \mathcal{I}}(A)$ for each $A \subseteq \mathbf{N}$ and $n \in \mathbf{N}$. %However, if the sequence $\mu$ has some additional properties

Note that, 
if $\mathcal{I}$ is an ideal on $\mathbf{N}$, then there exists a smooth sequence of submeasures $\mu$ such that $\mathcal{I}=\mathcal{Z}_\mu$: indeed, it is sufficient to set $\mu_n(A)$ equal to the characteristic function $\bm{1}_{\mathcal{P}(\mathbf{N})\setminus \mathcal{I}}(A)$ for each $A \subseteq \mathbf{N}$ and $n \in \mathbf{N}$.

If, in addition, $(I_n)$ is a partition of $\mathbf{N}$ into finite nonempty sets 
%intervals 
%and $\lim_n \sup_{k} \mu_n(\{k\})=0$ 
then $\mathcal{Z}_\mu$ is a \emph{generalized density ideal}, as introduced by Farah in \cite[Section 2.10]{MR1988247}, cf. also \cite{MR2254542}. 
%In such case, it is well known that $\mathcal{Z}_\mu$ is tall if and only if 
%\begin{equation}\label{eq:tallequiv}
%\lim_{n\to \infty} \, \sup_{k \in \mathbf{N}} \,\mu_n(\{k\})=0.
%\end{equation}
Recall that every generalized density ideal is an analytic P-ideal: indeed, $\mathcal{Z}_\mu$ coincides with $\mathrm{Exh}(\varphi_\mu)$, where $\varphi_\mu$ is the lower semicontinuous submeasure $\sup_k \mu_k$. The class of generalized density ideals is very rich, including for example all Erd{\H o}s--Ulam ideals (among others, $\mathcal{Z}$), the Fubini product $\emptyset\times \mathrm{Fin}$, simple density ideals \cite{MR3391516}, and ideals defined in \cite{MR1169042} by Louveau and Veli\v{c}kovi\'{c}, cf. \cite[Section 2]{Leo18meager} and references therein.

\begin{defi}\label{def:Istatconv}
Let $\mathcal{I}$ be an ideal and $\mu=(\mu_n)$ be a smooth sequence of submeasures. A sequence $(x_n)$ taking values in a Hausdorff topological space $X$ is said to be $(\mathcal{I},\mu)$\emph{-convergent} to $\ell$, shortened with $x_n \to_{(\mathcal{I},\,\mu)} \ell$, if 
$$
\left\{n \in \mathbf{N}: \mu_n(\{k \in \mathbf{N}: x_k \notin U\}) \notin V\right\} \in \mathcal{I}
$$
for each neighborhood $U$ of $\ell \in X$ and $V$ of $0 \in \mathbf{R}$. 
\end{defi}
 
In other words, the sequence $(x_n)$ is $(\mathcal{I},\mu)$-convergent to $\ell$ if and only if 
%the sequence $(\mu_n(\{k \in \mathbf{N}: x_k \notin U\}))$ is $\mathcal{I}$-convergent to $0$ 
$$\mu_n(\{k \in \mathbf{N}: x_k \notin U\})\to_{\mathcal{I}} 0$$  
%$$
%\mu_n(\{k \in \mathbf{N}: x_k \notin U\})\to_{\mathcal{I}} 0
%$$ 
for each neighborhood $U$ of $\ell$. Moreover, it is clear that $x_n \to_{(\mathrm{Fin},\,\mu)} \ell$ if and only if $x_n \to_{\mathcal{Z}_\mu} \ell$. This observation is generalized in Theorem \ref{thm:marek} below.

Hereafter, let $\lambda=(\lambda_n)$ be the sequence of uniform probability measures on $\mathbf{N}\cap [1,n]$, that is,
\begin{equation}\label{eq:uniformdistribution}
\lambda_n(A)=\frac{|A\cap [1,n]|}{n}
\end{equation}
for all $n \in \mathbf{N}$ and $A\subseteq \mathbf{N}$. Then it is easy to see that, for each ideal $\mathcal{I}$, $(\mathcal{I},\lambda)$-convergence corresponds with $\mathcal{I}$-statistical convergence defined in \eqref{eq:Istatconv}. 
Note that $(\mathcal{I},\mu)$-convergence includes also 
the case of $\mathcal{I}$-lacunary statistical convergence where each $\mu_n$ is the uniform probability measure on $\mathbf{N} \cap [a_n,a_{n+1})$ such that $(a_n)$ is an increasing sequence of positive integers for which $a_{n+1}-a_n \to \infty$, cf. \cite[Definition 6]{MR2803700}. 

We are ready to state our main results (all the proofs are given in Section \ref{sec:proofs}). 
Let us start with an equivalence with the classical notion of ideal convergence.
\begin{thm}\label{thm:marek}
Let $\mathcal{I}$ be an ideal and $\mu$ be a smooth sequence of submeasures. Then there exists a 
unique ideal 
%\textup{(}necessarily unique\textup{)} ideal 
%$\mathcal{J}=\mathcal{J}(\mathcal{I},\mu)$ such that $(\mathcal{I},\mu)$-convergence coincides with $\mathcal{J}$-convergence.
$\mathcal{J}=\mathcal{J}(\mathcal{I},\mu)$ such that $(\mathcal{I},\mu)$-convergence coincides with $\mathcal{J}$-convergence. In addition, $\mathcal{J}$ is proper if and only if $\mu_n(\mathbf{N}) \not\to_{\mathcal{I}} 0$.
%In addition, $\mathcal{J}$ is Borel if $\mathcal{I}$ is Borel and each $\mu_n$ is lower semicontinuous.
\end{thm}

To be precise, we say that 
$(\mathcal{I},\mu)$-convergence "coincides" with $\mathcal{J}$-convergence if, for some Hausdorff space $X$ with at least two points, every sequence $(x_n)$ taking values in $X$ is $(\mathcal{I},\mu)$-convergent to $\ell \in X$ if and only if $(x_n)$ is $\mathcal{J}$-convergent to $\ell$.

At this point, one may ask whether there is some relationship between the pair $(\mathcal{I},\mu)$ and the ideal $\mathcal{J}(\mathcal{I},\mu)$ in Theorem \ref{thm:marek}. First of all, we prove that, under some mild conditions, $\mathcal{J}(\mathcal{I},\mu)$ has the same topological complexity of $\mathcal{I}$.
\begin{thm}\label{thm:marekborel}
Let $\mathcal{I}$ be an ideal and $\mu$ be a smooth sequence of lower semicontinuous submeasures. Then the ideal $\mathcal{J}(\mathcal{I},\mu)$ is Borel \textup{[}analytic, coanalytic, respectively\textup{]} whenever $\mathcal{I}$ is Borel \textup{[}analytic, coanalytic, resp.\textup{]}.
%Then there exist continuous functions $f_{m,k}: \mathcal{P}(\mathbf{N})\to \mathcal{P}(\mathbf{N})$ for each $m,k \in \mathbf{N}$ such that 
%$$
%\mathcal{J}(\mathcal{I},\mu)=\bigcap_{m \in \mathbf{N}}\bigcup_{n \in \mathbf{N}}f_{m,k}^{-1}\,[\mathcal{I}].
%$$
%In particular, the ideal $\mathcal{J}(\mathcal{I},\mu)$ is Borel whenever $\mathcal{I}$ is Borel.}
\end{thm}

It turns out that if the pair $(\mathcal{I},\mu)$ is "sufficiently nice" then 
we can find the ideal $\mathcal{J}(\mathcal{I},\mu)$ explicitly. 
%$\mathcal{J}(\mathcal{I},\mu)=\mathcal{Z}_\mu$.} 
To this aim, we need the following definitions.
\begin{defi}\label{defi:alphathick}
Given a real $\alpha>0$, we say that an ideal $\mathcal{I}$ is $\alpha$\emph{-thick} provided that $A\notin \mathcal{I}$ whenever there exist a real $c>0$ and infinitely many $n \in \mathbf{N}$ such that $\mathbf{N} \cap [\,n,n+cn^\alpha] \subseteq A$.
\end{defi}
%\textcolor{red}{[Add examples here or inside the corollaries?]}

It turns out that $\mathcal{Z}$ is $1$-thick, cf. the proof of Corollary \ref{cor:idealstat}. 
On the other hand, if $\mathcal{I}$ is maximal ideal, then $\mathcal{I}$ is not $\alpha$-thick, for every $\alpha>0$. Indeed, exactly one among the sets $A:=\mathbf{N} \cap \bigcup_{n} [a_{2n}, a_{2n+1}]$ and $A^c$ belongs to $\mathcal{I}$, where $a_1:=1$ and $a_{n+1}:=2^{a_n}$ for all $n \in \mathbf{N}$, and such a set contains infinitely many intervals of the type $\mathbf{N}\cap [n,n+cn^\alpha]$, for each $c>0$.

\begin{defi}\label{defi:alphaflat}
Given a real $\alpha>0$, we say that a sequence of submeasures $\mu=(\mu_n)$ is $\alpha$\emph{-flat} provided that, for each $A\subseteq \mathbf{N}$, there exists a real $d=d(A)>0$ such that 
$$%\begin{equation}\label{eq:additionalcondtionsmooth}
|\mu_{n+1}(A)-\mu_n(A)| \le d/n^\alpha
$$%\end{equation}
for all $n \in \mathbf{N}$.
\end{defi}
%Note that if $\mu$ is $\alpha$-flat then it satisfies \eqref{eq:tallequiv}, hence $\mathcal{Z}_\mu$ is tall. 
%

It is easy to show that the sequence $\lambda$ defined in \eqref{eq:uniformdistribution} is $1$-flat, cf. the proof of Corollary \ref{cor:idealstat} for details. 
More generally, for each $\alpha>0$, a family of $\alpha$-flat sequences is given as follows. Suppose that $\mu$ is smooth and each $\mu_n$ is a probability measure supported on $I_n:=\mathbf{N}\cap [1,\iota_n]$, where $(\iota_n)$ is an increasing sequence in $\mathbf{N}$ such that:
\begin{enumerate}[label={\rm (\roman{*})}]%[label={\rm (\textsc{d}\arabic{*})}]
\item $\iota_n \le bn^\beta$ for all $n$ and some $b,\beta>0$;
\item $\iota_{n+1}-\iota_n \le cn^{\gamma}$ for all $n$ and some $c,\gamma>0$;
\item $|\mu_{n+1}(\{k\})-\mu_n(\{k\})|\le d/n^{\delta}$ for all $n$, all $k\le \iota_n$, and some $d,\delta>0$;
\item $\mu_n(\{k\}) \le e/n^{\eta}$ for all $n$, all $k>\iota_n$, and some $e,\eta>0$;
\item $\alpha \le \min\{\delta-\gamma,\eta-\beta\}$.
\end{enumerate}
Then it is routine to check that $\mu$ is $\alpha$-flat. 
Similarly, if we assume for simplicity that $\iota_n=n$ for all $n$, then $\mathcal{Z}_\mu$ is $\alpha$-thick, for some $\alpha \in (0,1)$, provided that $\mu_n(\mathbf{N} \cap [n-cn^\alpha,n]) \not\to 0$ for all $c>0$. 
It is worth to remark that, in such cases, the ideal $\mathcal{Z}_\mu$ corresponds to the ideal generated by the nonnegative regular matrix $R=\{r_{n,k}: n,k \in \mathbf{N}\}$, where $r_{n,k}:=\mu_n(\{k\})$, cf. e.g. \cite[Section 2]{DFT2019}.

With these premises, we can state the following characterization.
\begin{thm}\label{thm:character}
%Let $\mu=(\mu_n)$ and $\nu=(\nu_n)$ be
Let $\nu$ and $\mu$ be two smooth sequences of submeasures such that $\mathcal{Z}_\nu$ is $\alpha$-thick and $\mu$ is $\alpha$-flat, for some $\alpha \in (0,1]$. 
%
%
%%, for some $\alpha>0$, the ideal $\mathcal{Z}_\mu$ is $\alpha$-thick and
%\begin{equation}\label{eq:additionalcondtionsmooth}
%\sup_{A\subseteq \mathbf{N}}\,\left|\nu_{n}(A)-\nu_{n+1}(A)\right|=\mathcal{O}\left(1/n^\alpha\right)%\left(\frac{1}{n^\alpha}\right).}
%\end{equation}
%for some $\alpha\in (0,1]$ and $\mathcal{Z}_\mu$ is $\alpha$-thick. 
%%for some $\alpha>0$ and $\mathcal{Z}_\mu$ is $\min(\alpha,1)$-thick. 
%
Then $\mathcal{J}(\mathcal{Z}_\nu,\mu)=\mathcal{Z}_\mu$, that is, $(\mathcal{Z}_\nu,\mu)$-convergence coincides with $\mathcal{Z}_\mu$-convergence.
\end{thm}

%%\begin{comment}
%%\textcolor{blue}{We restricted our attention to the case $\alpha \in (0,1]$ because, in the opposite, 
%%%$\zeta(\alpha)=\sum_n \frac{1}{n^\alpha}$ would be finite, hence 
%%if $\mathcal{I}=\mathcal{Z}_\mu$ and $\mu_{n_0}(A)>r:=\sum_n \frac{1}{n^\alpha}$ (which is finite) for some $A\subseteq \mathbf{N}$ and $n_0 \in \mathbf{N}$, then $\liminf_n \mu_n(A)\ge r>0$ which implies $A \notin \mathcal{Z}_\mu$; however, in practical examples the value of a single submeasure $\mu_{n_0}$ should not be sufficient to establish whether a set belongs to $\mathcal{Z}_\mu$ or not.}
%%\textcolor{blue}{(I actually don't get what I was meaning here)}
%%\end{comment}
%
%The following corollaries are almost immediate:

Roughly, Theorem \ref{thm:character} states that if the pair $(\mathcal{I},\mu)$ is sufficiently nice, then
\begin{equation}\label{eq:tauberian}
\mu_n(A) \to_{\mathcal{I}} 0 \,\,\,\,\text{ implies }\,\,\,\,\mu_n(A) \to 0
\end{equation}
for all $A\subseteq \mathbf{N}$. 
%, the converse being clear. 
At the point, if $A$ is fixed, the real sequence $(x_n)$ defined by $x_n:=\mu_n(A)$ is arbitrary, 
though nonnegative. Hence, in the case $\mathcal{I}=\mathcal{Z}$ and $|x_{n+1}-x_n| \le d/n$ for all $n$ and some $d>0$ (which corresponds to $1$-flatness of the sequence $\mu$ relative to $A$), the claim \eqref{eq:tauberian} can be rewritten as
$$
x_n \to_{\mathcal{Z}} 0 \,\,\,\,\text{ implies }\,\,\,\,x_n \to 0.
$$
Indeed, this is a classical result of Fridy, see \cite[Theorem 3]{MR816582}. Here, he also proves that the Tauberian condition $|x_{n+1}-x_n| \le d/n$ is best possible. 
This has been soon extended by Maddox, in the context of strong summability for slowly oscillating sequences, see \cite{MR938459, MR1002541}. 
A quite different Tauberian condition for Borel summability (related to $\nicefrac{1}{2}$-flatness) can be found in \cite{MR1653457}. 
Finally, there are related results for statistically slowly oscillating sequences \cite{MR2324333, MR1941785, MR2009641, MR2079327} and for sequences which satisfy a gap Tauberian condition \cite{MR3906365, MR1989684}.

%\cite{MR2324333, 
%MR3906365, 
%MR1659877, MR1653457, MR1989684, 
%MR938459, MR1002541, MR1941785, MR2009641, MR2079327}.}

As an application of Theorem \ref{thm:character}, we obtain a sufficient condition for the equivalence between statistical convergence and $\mathcal{I}$-statistical convergence.
\begin{cor}\label{cor:idealstat}
Let $\mathcal{I}$ be an ideal such that $\mathcal{I}\subseteq \mathcal{Z}$. Then $\mathcal{I}$-statistical convergence coincides with statistical convergence.
\end{cor}

Since $\mathcal{Z}$ is the ideal generated the upper asymptotic density $\mathrm{d}^\star$ defined by 
\begin{equation}\label{eq:dstar}
\mathrm{d}^\star(A):=\limsup_{n\to \infty}\frac{|A \cap [1,n]|}{n}
\end{equation}
for all $A\subseteq \mathbf{N}$ (that is, $\mathcal{Z}=\{A\subseteq \mathbf{N}: \mathrm{d}^\star(A)=0\}$), it follows that Corollary \ref{cor:idealstat} applies to all ideals $\mathcal{I}$ of the type $\{A\subseteq \mathbf{N}: \mu^\star(A)=0\}$, where $\mu^\star$ is an "upper density" on $\mathbf{N}$, in the sense of \cite{LeoTri}, such that $\mathrm{d}^\star \le \mu^\star$ pointwise. In particular, possible choices for $\mu^\star$ are: the upper Banach density, the upper analytic density, the upper P\'{o}lya density, the upper Buck density, together with all upper $\alpha$-densities with $\alpha\ge 0$ (see \cite{LeoTri} for details; cf. also \cite{DiNasso17} for the relationship between ideals and "abstract densities").

In addition, as a special instance of Corollary \ref{cor:idealstat}, we have:
\begin{cor}\label{cor:idealstat2}
$\mathcal{I}$-statistical convergence coincides with statistical convergence if $\mathcal{I}=\mathcal{I}_{1/n}$ or $\mathcal{I}=\emptyset \times \mathrm{Fin}$.
\end{cor}

%\begin{cor}\label{cor:idealstat3}
%$$-statistical convergence coincides with statistical convergence.
%\end{cor}

%The proof of Theorem \ref{thm:character} follows 
%Proofs follow in Section \ref{sec:proofs}.

Lastly, we show that the conclusion of Corollary \ref{cor:idealstat} cannot be strenghtened to the whole class of ideals $\mathcal{I}$. Indeed, this is \emph{never} the case if $\mathcal{I}$ is maximal.
\begin{thm}\label{thm:maximal}
Let $\mathcal{I}$ be a maximal ideal. 
Then $\mathcal{I}$-statistical convergence does not coincide with statistical convergence.
\end{thm}

To conclude, note that the definition of $(\mathcal{I},\mu)$-convergence \emph{depends} on the choice of the sequence $\mu$. Indeed, it is possible that $(\mathcal{I},\mu)$-convergence does not coincide with $(\mathcal{I},\nu)$-convergence, where $\nu$ is another smooth sequence of submeasures for which $\mathcal{Z}_\mu=\mathcal{Z}_{\nu}$.
\begin{example}\label{example2}
Let $x$ be the sequence defined by $x_n=\bm{1}_A(n)$ for all $n \in \mathbf{N}$, where $A:=\bigcup_{n\ge 1} [\,(2n)!,(2n+1)!\,]$. 
Moreover, set $\mu_n=\lambda_n$ and $\nu_n=\mathrm{d}^\star$, as defined in \eqref{eq:uniformdistribution} and \eqref{eq:dstar}, respectively, for all $n \in \mathbf{N}$, so that $\mathcal{Z}_\mu=\mathcal{Z}_\nu=\mathcal{Z}$. 
Note that $\mathrm{d}^\star(A)=\mathrm{d}^\star(A^c)=1$ (in particular, $x$ is not statistically convergent). 
Set $A_m:=\{n\in \mathbf{N}: \lambda_n(A)\ge \nicefrac{1}{m}\}$ for each $m \in \mathbf{N}$. 
Then $(A_m)$ is an increasing sequence of sets and there exists a maximal ideal $\mathcal{I}$ containing all the $A_m$'s. It follows that, for every $\varepsilon>0$, the set $\{n\in \mathbf{N}: \lambda_n(n)\ge \varepsilon\}$ is contained in some $A_m \in \mathcal{I}$, so that $\lambda_n(A) \to_{\mathcal{I}} 0$. Therefore $x_n \to_{(\mathcal{I},\,\mu)} 1$. 
On the other hand, 
\begin{displaymath}
\begin{split}
\{n \in \mathbf{N}: \nu_n(\{k \in \mathbf{N}: |x_k-1|&\ge \nicefrac{1}{2}\})\ge \nicefrac{1}{2}\}\\ &=\{n \in \mathbf{N}: \nu_n(\{k \in \mathbf{N}: x_k=0\})\ge \nicefrac{1}{2}\}\\
&=\{n \in \mathbf{N}: \mathrm{d}^\star(A^c)\ge \nicefrac{1}{2}\}=\mathbf{N} \notin \mathcal{I},
\end{split}
\end{displaymath}
hence $x_n \not\to_{(\mathcal{I},\,\nu)}1$.
\end{example}

We leave as open questions for the interested reader to "characterize" the class of ideals $\mathcal{I}$ for which $\mathcal{I}$-statistical convergence coincides with statistical convergence and to establish whether Theorem \ref{thm:maximal} 
holds for nonmeasurable ideals or those without the Baire property.
%holds for ideals without the Baire property.

%%%%%%%%%%%%%%%%%%%%%%%%%%%%%%%%%%%%%%%%%%%%%%%%%%%%%%%%%%%%%%%%%%%%%%%%%%%%%%%%%%%%%%%%%
\section{Proofs}\label{sec:proofs}

Before we start proving our results, we state the next lemma (which is straightforward, we omit details):
\begin{lem}\label{lem:basicimplications}
Fix ideals $\mathcal{I},\mathcal{J}$ on $\mathbf{N}$, let $\mu,\nu$ be two smooth sequences of submeasures, and fix $\alpha,\beta>0$. Then:
\begin{enumerate}[label={\rm \textup{(}\roman*\textup{)}}]
\item \label{item:l1} $(\mathcal{I},\mu)$-convergence implies $(\mathcal{J},\mu)$-convergence, provided that $\mathcal{I}\subseteq \mathcal{J}$;
\item \label{item:l2} $(\mathcal{I},\mu)$-convergence implies $(\mathcal{I},\nu)$-convergence, 
provided that, for all $A\subseteq \mathbf{N}$, it holds $\nu_n(A) \le \mu_n(A)$ for all sufficiently large $n$;
\item \label{item:l3} $\mathcal{I}$ is $\alpha$-thick implies that $\mathcal{J}$ is $\alpha$-thick, provided that $\mathcal{J}\subseteq \mathcal{I}$;
\item \label{item:l4} $\mathcal{I}$ is $\alpha$-thick implies that $\mathcal{I}$ is $\beta$-thick, provided that $\alpha\le \beta$; 
\item \label{item:l5} $\mu$ is $\alpha$-flat implies that $\mu$ is $\beta$-flat, provided that $\beta\le \alpha$. 
\end{enumerate}
\end{lem}

Thus, let us start with the proof of Theorem \ref{thm:marek}.
\begin{proof}[Proof of Theorem \ref{thm:marek}]
Given the pair $(\mathcal{I},\mu)$, define the family
\begin{equation}\label{eq:JImu}
\mathcal{J}=\mathcal{J}(\mathcal{I},\mu):=\left\{A\subseteq \mathbf{N}: \mu_n(A)\to_{\mathcal{I}} 0\right\}.
\end{equation}
It is clear that $\mathcal{J}$ is closed under subsets. Moreover, $\mathcal{J}$ is closed under finite unions;  indeed, for all $A,B \in \mathcal{J}$ and $\varepsilon>0$, we have $\{n\in \mathbf{N}: \mu_n(A)\ge \varepsilon/2\} \in \mathcal{I}$ and $\{n\in \mathbf{N}: \mu_n(B)\geq\varepsilon/2\} \in \mathcal{I}$; hence 
$$
\{n: \mu_n(A\cup B)\ge \varepsilon\}\subseteq \{n: \mu_n(A)\ge \varepsilon/2\} \cup \{n: \mu_n(B)\ge \varepsilon/2\} \in \mathcal{I},
$$
so that $A\cup B \in \mathcal{J}$. Since $\mu$ is a smooth sequence of submeasures, we have $\lim_{n\to \infty}\mu_n(\{k\})\to 0$ for all $k \in \mathbf{N}$, by \ref{item:s2}. It follows that $\mu_n(\{k\}) \to_{\mathcal{I}} 0$, which implies that $\mathrm{Fin}\subseteq \mathcal{J}$. This shows that $\mathcal{J}$ is an ideal on $\mathbf{N}$ and, in addition, $\mathcal{J}$ is proper if and only if $\mu_n(\mathbf{N}) \not\to_{\mathcal{I}} 0$. 
At this point, let us prove that $(\mathcal{I},\mu)$-convergence coincides with $\mathcal{J}$-convergence. Let $X$ be an Hausdorff space with at least two points, let us say $a$ and $b$, and let $(x_n)$ be a sequence in $X$ which is $(\mathcal{I},\mu)$-convergent to some $\ell \in X$, that is, $\mu_n(\{k \in \mathbf{N}: x_k\notin U\}) \to_{\mathcal{I}} 0$ for each neighborhood $U$ of $\ell$. By the definition of $\mathcal{J}$ in \eqref{eq:JImu}, this is equivalent to $\{k\in \mathbf{N}: x_k \notin U\} \in \mathcal{J}$ for each neighborhood $U$ of $\ell$, i.e., $x_n\to_{\mathcal{J}} \ell$. 
Finally, let us suppose for the sake of contradiction that there exists another ideal $\mathcal{J}^\prime\neq \mathcal{J}$ such that $(\mathcal{I},\mu)$-convergence coincides with $\mathcal{J}^\prime$-convergence. In particular, there exists $A \in \mathcal{J} \triangle \mathcal{J}^\prime$ and $\mathcal{J}$-convergence coincides with $\mathcal{J}^\prime$-convergence. Let $(x_n)$ be the sequence defined by 
$x_n=a$ if $n \in A$ and $x_n=b$ otherwise. 
It follows that exactly one of the conditions $x_n\to_{\mathcal{J}} b$ and $x_n\to_{\mathcal{J}^\prime}b$ is true. This shows that $\mathcal{J}$ is unique, completing the proof.
\end{proof}

Note that the ideal $\mathcal{J}(\mathcal{I},\mu)$ defined in \eqref{eq:JImu} corresponds to $\mathcal{Z}_\mu$ if $\mathcal{I}=\mathrm{Fin}$.

\begin{proof}[Proof of Theorem \ref{thm:marekborel}]
Let us rewrite the ideal $\mathcal{J}(\mathcal{I},\mu)$ defined in \eqref{eq:JImu} as follows:
\begin{displaymath}
\mathcal{J}(\mathcal{I},\mu)=\left\{A\subseteq \mathbf{N}: \forall m \in \mathbf{N}, f_m(A)\in \mathcal{I}\right\},
\end{displaymath}
where, for each $m \in \mathbf{N}$, $f_m: \mathcal{P}(\mathbf{N}) \to \mathcal{P}(\mathbf{N})$ is the function defined by 
$$
f_m(A):=\{n \in \mathbf{N}: \mu_n(A)>1/m\}
$$
for all $A\subseteq \mathbf{N}$. At this point, by the lower semicontinuity of each $\mu_n$, we obtain 
$
\textstyle f_m(A)=\bigcup_{k\in \mathbf{N}}f_{m,k}(A),
$ 
where, for each $k \in \mathbf{N}$, $f_{m,k}: \mathcal{P}(\mathbf{N}) \to \mathcal{P}(\mathbf{N})$ is the function defined by 
$$
f_{m,k}(A):=\{n \in \mathbf{N}: \mu_n(A \cap [1,k])>1/m\}
$$
for all $A\subseteq \mathbf{N}$ and $m \in \mathbf{N}$. 
Therefore
$$
\mathcal{J}(\mathcal{I},\mu)=\left\{A\subseteq \mathbf{N}: \forall m \in \mathbf{N}, \exists k \in \mathbf{N}, f_{m,k}(A)\in \mathcal{I}\right\}=\bigcap_{m \in \mathbf{N}} \bigcup_{k \in \mathbf{N}} f_{m,k}^{-1}\,[\mathcal{I}].
$$
The claim follows by noting that $f_{m,k}$ is continuous and that the continuous preimage of Borel [analytic, coanalytic, respectively] sets is Borel [analytic, coanalytic, resp.], cf. e.g. \cite{MR1619545}.
\end{proof}

To prove the next result, we need the following intermediate lemma.
\begin{lem}\label{lem:implication}
Let $\mu$ and $\nu$ be two smooth sequences of submeasures. Then $\mathcal{Z}_\mu$-convergence implies $(\mathcal{Z}_\nu,\mu)$-convergence.
\end{lem}
\begin{proof}
Let $(x_n)$ be a sequence in a Hausdorff space $X$, fix $\ell \in X$, and suppose that 
$x_n \not\to_{(\mathcal{Z}_\nu,\mu)} \ell$. 
Then there exist a neighborhood $U$ of $\ell$ and a real $\varepsilon>0$ such that
$$
\left\{n \in \mathbf{N}: \mu_n(\{k \in \mathbf{N}: x_k \notin U\}) \ge \varepsilon\right\}\notin \mathcal{Z}_\nu,
$$
that is,
$$
\limsup_{t \to \infty} \nu_t\left(\left\{n \in \mathbf{N}: \mu_n(\{k \in \mathbf{N}: x_k \notin U\}) \ge \varepsilon\right\}\right)>0.
$$
It follows that there exists a strictly increasing sequence of positive integers $(t_m)$ and a real $\delta>0$ such that
$$
\nu_{t_m}\left(\left\{n \in I_{t_m}: \mu_n(\{k \in I_n: x_k \notin U\}) \ge \varepsilon\right\}\right)\ge \delta
$$ 
for all $m \in \mathbf{N}$. At this point, fix an integer $n_1 \in I_{t_1}$ such that $\mu_{n_1}(\{k \in I_{n_1}: x_k \notin U\}) \ge \varepsilon$ and define recursively a sequence $(n_h)$ of positive integers as follows: for each $h \in \mathbf{N}$, let $n_{h+1}$ be an integer greater than $n_h$ such that $\mu_{n_{h+1}}(\{k \in I_{n_{h+1}}: x_k \notin U\}) \ge \varepsilon$; note that such an integer exists because the sequence $(t_m)$ is infinite and $F:=\{n_1,\ldots,n_h\} \in \mathrm{Fin} \subseteq \mathcal{Z}_\nu$ so that the sequence $\nu_n(F)$ converges to $0$ and, in particular, it is smaller than $\delta$ whenever $n$ is sufficiently large. It follows that 
$$
\mu_{n_{h}}(\{k \in I_{n_{h}}: x_k \notin U\}) \ge \varepsilon
$$ 
for all $h \in \mathbf{N}$, which implies $\{k \in \mathbf{N}: x_k \notin U\} \notin \mathcal{Z}_\mu$. Therefore $(x_n)\not\to_{\mathcal{Z}_\mu}\ell$.
\end{proof}

We proceed now to the proof of Theorem \ref{thm:character}.
\begin{proof}[Proof of Theorem \ref{thm:character}]
On the one hand, thanks to Lemma \ref{lem:implication}, $\mathcal{Z}_\mu$-convergence implies $(\mathcal{Z}_\nu,\mu)$-convergence. 
Conversely, let $(x_n)$ be a sequence in a Hausdorff space $X$ which is not $\mathcal{Z}_\mu$-convergent to $\ell \in X$; hence there exist a neighborhood $U$ of $\ell$, a real $\varepsilon>0$, and an increasing sequence of positive integers $(n_t)$ such that
\begin{equation}\label{eq:inequality0flat}
\mu_{n_t}(\{k \in \mathbf{N}: x_k \notin U\})\ge \varepsilon
\end{equation}
for all $t \in \mathbf{N}$. 
At this point, fix $S\subseteq \mathbf{N}$. Since the sequence $\mu$ is $\alpha$-flat, we obtain that there exists $d=d(S)>0$ such that $f_S(n)\le d/n^\alpha$ for all $n \in \mathbf{N}$,  
where $f_S(n):=\left|\mu_{n}(S)-\mu_{n+1}(S)\right|$ for each $n \in \mathbf{N}$.

\begin{claim}\label{claim:estimatesum}
Fix $S\subseteq \mathbf{N}$. There exists a constant $\kappa=\kappa(\alpha,d(S))>0$ such that 
$$
\sum_{i=1}^{\lfloor cn^{\alpha}\rfloor }f_S(n+i) \le \kappa c.%:=\max\left\{\frac{cd}{1-\alpha}, (c+1)d\right\}
$$
for all $n\in \mathbf{N}$ and reals $c>0$.
\end{claim}
\begin{proof}
First of all, for every $n\in \mathbf{N}$ and $c>0$, we have the following upper bound
\begin{equation}\label{eq:upperbound}
\sum_{i=1}^{\lfloor cn^{\alpha}\rfloor }f_S(n+i)  \le \sum_{i=1}^{\lfloor cn^\alpha\rfloor}\frac{d}{(n+i)^\alpha} \le \int_{0}^{{cn^\alpha}} \frac{d}{(n+t)^\alpha}\,\mathrm{d}t.
\end{equation}
Hence, if $\alpha \in (0,1)$, we have
\begin{displaymath}
\begin{split}
\sum_{i=1}^{\lfloor cn^{\alpha}\rfloor }f_S(n+i) &\le \frac{d}{1-\alpha}\left((n+cn^\alpha)^{1-\alpha}-n^{1-\alpha}\right)\\
&\le \frac{d}{1-\alpha}\,n^{1-\alpha}\left(\left(1+cn^{\alpha-1}\right)^{1-\alpha}-1\right)\\
&\le \frac{d}{1-\alpha}\,n^{1-\alpha}\cdot cn^{\alpha-1} = \frac{cd}{1-\alpha}.
\end{split}
\end{displaymath}
Similarly, if $\alpha=1$, it follows by \eqref{eq:upperbound} that
\begin{displaymath}
\begin{split}
\sum_{i=1}^{\lfloor cn^{\alpha}\rfloor }f_S(n+i) &\le  d \log \left(\frac{n+cn}{n}\right)=d(\log(1+c))\le cd,
\end{split}
\end{displaymath}
which completes the proof.
\end{proof}

\begin{claim}\label{claim2}
Fix $S\subseteq \mathbf{N}$. Then for all $\delta>0$, there exist $c>0$ and $n_0 \in \mathbf{N}$ such that
\begin{equation}\label{eq:claim2}
%\forall \delta>0, \exists c>0, \exists n_0 \in \mathbf{N}, \forall n\ge n_0, \forall m \in %[1,cn^\alpha],\,\,\,\,\,\, 
|\mu_{n+m}(S)-\mu_n(S)| \le \delta
\end{equation}
for all integers $n\ge n_0$ and all integers $m \in [0,cn^\alpha]$.
\end{claim}
\begin{proof}
With the same notation above, it follows by Claim \ref{claim:estimatesum} that 
for all $n \in \mathbf{N}$, all $c>0$, and all integers $m \in [1,cn^\alpha]$, we have that
$$
|\mu_{n+m}(S)-\mu_n(S)| \le \sum_{i=0}^m f_S(n+i) \le f_S(n) + \kappa c \le \frac{d}{n^\alpha}+ \kappa c.
$$
At this point, the wanted inequality \eqref{eq:claim2} is obtained by choosing $c>0$ sufficiently small and $n_0 \in \mathbf{N}$ sufficiently large so that $d/n_0^\alpha+\kappa c \le \delta$.
\end{proof}

To conclude, choosing $\delta=\nicefrac{\varepsilon}{2}$ and $S=\{k \in \mathbf{N}: x_k \notin U\}$ in Claim \ref{claim2} and using inequality \eqref{eq:inequality0flat}, we obtain that there exist $c>0$ and $t_0 \in \mathbf{N}$ such that 
$$
\mu_{n_t+m}(S)\ge \mu_{n_t}(S)-\nicefrac{\varepsilon}{2} \ge \nicefrac{\varepsilon}{2}
$$
for all integers $t\ge t_0$ and all integers $m \in [\,0,cn_t^\alpha\,]$. Therefore 
$$
A:=\left\{n \in \mathbf{N}: \mu_n(S) \ge \nicefrac{\varepsilon}{2}\right\} \supseteq \bigcup_{t\ge t_0} (\mathbf{N}\cap [\,n_t, (1+c)n_t^\alpha\,]).
$$
Since $\mathcal{Z}_\nu$ is $\alpha$-thick, then $A\notin \mathcal{Z}_\nu$. This implies that $x_n\not\to_{(\mathcal{Z}_\nu,\mu)}\ell$.
\end{proof}

\begin{proof}[Proof of Corollary \ref{cor:idealstat}]
Let $\lambda$ be the smooth sequence of submeasures defined in \eqref{eq:uniformdistribution}. Thanks to Theorem \ref{thm:character}, it is  sufficient to show that $\mathcal{Z}=\mathcal{Z}_\lambda$ is $1$-flat and that $\mathcal{I}$ is $1$-thick. 
To this aim, note that, for all $A\subseteq \mathbf{N}$ and $n \in \mathbf{N}$, we have
$$
\left|\lambda_{n}(A)-\lambda_{n+1}(A)\right|=\left|\,\lambda_{n}(A)-\frac{\bm{1}_A(n+1)+n\lambda_n(A)}{n+1}\,\right|\le \frac{1}{n+1}\le \frac{1}{n}.
$$
Hence $\mathcal{Z}$ is $1$-flat. 
Moreover, we claim that $\mathcal{Z}$ is $1$-thick. Indeed, fix $A\subseteq \mathbf{N}$ such that there exist $c>0$ and infinitely many $n \in \mathbf{N}$ for which $\mathbf{N} \cap [\,n,(1+c)n\,] \subseteq A$. Then the upper asymptotic density of $A$ is at least $\nicefrac{c}{1+c}>0$. Hence $A\notin \mathcal{Z}$. 
Therefore $\mathcal{I}$ is $1$-thick by Lemma \ref{lem:basicimplications}\ref{item:l3}.
\end{proof}

\begin{proof}[Proof of Corollary \ref{cor:idealstat2}]
It is known $\mathcal{I}_{1/n} \subseteq \mathcal{Z}$, hence the claim follows by Corollary \ref{cor:idealstat}. Similarly, it is sufficient to prove that there exists a copy of the Fubini product $\mathcal{I}=\emptyset \times \mathrm{Fin}$ on $\mathbf{N}$ which is contained in $\mathcal{Z}$. Thus, note that $\mathcal{I}$ can be written as 
$$
\{A\subseteq \mathbf{N}: \forall k \in \mathbf{N}, \{n \in A: \upsilon_2(n)=k-1\} \in \mathrm{Fin}\},
$$
where $\upsilon_2(n)$ is the biggest exponent $m \in \mathbf{N}$ such that $2^{m}$ divides $n$. Fix $A \in \mathcal{I}$. Then, for each $k \in \mathbf{N}$, there exists a finite set $F=F(k) \subseteq \mathbf{N}$ such that $A\setminus F$ contains only multiples of $2^k$, so that the upper asymptotic density of $A$ is most $1/2^k$. By the arbitrariness of $k$, we conclude that $A \in \mathcal{Z}$. Therefore $\mathcal{I}\subseteq \mathcal{Z}$. 
\end{proof}

\begin{proof}[Proof of Theorem \ref{thm:maximal}]
Let us suppose that $\mathcal{I}$ is a maximal ideal on $\mathbf{N}$. 
Thanks to Theorem \ref{thm:marek}, there exists a unique ideal $\mathcal{J}$ such that  $(\mathcal{I},\lambda)$-convergence coincides with $\mathcal{J}$-convergence; in addition, 
$
\mathcal{J}:=\{A \subseteq \mathbf{N}: \lambda_n(A) \to_{\mathcal{I}} 0\},
$ 
see \eqref{eq:JImu}. To conclude, we show that $\mathcal{J} \neq \mathcal{Z}$ by proving that they have different topological complexities. 
To this aim, it is sufficient to prove that $\mathcal{J}$ is nonmeasurable (on the other hand, it is known that $\mathcal{Z}$ is a $F_{\sigma\delta}$-ideal).

Identifying each set $A\subseteq \mathbf{N}$ with the sequence $(\bm{1}_A(n): n \in \mathbf{N}) \in \{0,1\}^{\mathbf{N}}$, we can rewrite $\mathcal{J}$ as
$$
%\textstyle \mathcal{J}=\left\{x \in \{0,1\}^{\mathbf{N}}: \frac{1}{n}\sum_{k=1}^n a_i \to_{\mathcal{I}} 0\right\}.
\textstyle \mathcal{J}=\left\{x \in \{0,1\}^{\mathbf{N}}: \frac{1}{n}\sum_{i\le n} x_i \to_{\mathcal{I}} 0\right\}.
$$
\begin{claim}\label{claim:existencean}
There exists an increasing sequence $(a_n)$ in $\mathbf{N}$ such that $a_n/a_{n+1} \to 0$  as $n\to \infty$ and $\bigcup_{n\in \mathbf{N}}A_{3n-1} \not\in \mathcal{I}$, where $A_n:=\mathbf{N}\cap (a_{n-1},a_n]$ and $a_0:=0$.
\end{claim}
\begin{proof}
Fix an increasing sequence $(a_n)$ in $\mathbf{N}$ such that $a_n/a_{n+1} \to 0$ and, for each $i \in \{0,1,2\}$, define $R_i:=\bigcup_{n\equiv i\bmod{3}}A_n$. Since $\mathcal{I}$ is maximal, there exists a unique $i_\star \in \{0,1,2\}$ such that $R_{i_\star} \notin \mathcal{I}$. If $i_\star=2$ we are done. If $i_\star=0$ [$i_\star=1$, respectively], just delete one element [two elements, resp.] from the sequence.% $(a_n)$.}
\end{proof}
At this point, fix an increasing sequence $(a_n)$ as in Claim \ref{claim:existencean} and define the function $\Lambda: \{0,1\}^{\mathbf{N}} \to \{0,1\}^{\mathbf{N}}$ by 
$$
\Lambda(x)=(x_{a_{3n-1}}: n \in \mathbf{N})
$$
for each sequence $\{0,1\}^{\mathbf{N}}$. Moreover, define the function $f: \mathbf{N}\to \mathbf{N}$ such that 
$f(n)=k$ whenever $n \in A_{k}$, 
and set 
\begin{displaymath}
\begin{split}
\mathcal{F}:=\{x \in \{0,1\}^{\mathbf{N}}: x_i=x_j &\text{ whenever }f(i)=f(j), \\
&\text{ and }x_{a_{3n-2}}=x_{a_{3n-1}} \text{ and } x_{a_{3n}}=0 \text{ for all }n \in \mathbf{N}\},
\end{split}
\end{displaymath}
that is, $\mathcal{F}$ is the sequence of $\{0,1\}$-valued sequences which are constant on all intervals $A_{3n}$ and $A_{3n-1}\cup A_{3n-2}$, 
with value $0$ on all the $A_{3n}$'s. 
Then the restriction of $\Lambda$ on $\mathcal{F}$ is an homeomorphism $\mathcal{F} \to \{0,1\}^{\mathbf{N}}$. Since $\mathcal{F}$ is closed, it is sufficient to show that $\Lambda(\mathcal{J}\cap \mathcal{F})$ is not measurable.

\begin{claim}\label{claim:estimatefactorials}
$\lim_{n\to \infty} \left(x_{a_{n}}-\frac{1}{a_{n}}\sum_{i\le a_{n}}x_i\right)=0$ for all $x \in \mathcal{F}$.
\end{claim}
\begin{proof}
Considering that $x_i=x_j$ whenever $f(i)=f(j)$ and that $a_{n}/a_{n+1}\to 0$ as $n\to \infty$, we have that
\begin{displaymath}
\begin{split}
\frac{1}{a_{n}}\sum_{i\le a_{n}}x_i&=\frac{1}{a_{n}}\sum_{i\le n}x_{a_i}|A_i|\\ 
&=\left(1-\frac{a_{n-1}}{a_{n}}\right)x_{a_{n}}+\frac{1}{a_{n}}\sum_{i\le n-1}x_{a_i}|A_i|.
\end{split}
\end{displaymath}
Therefore
\begin{displaymath}
\begin{split}
\left|x_{a_{n}}-\frac{1}{a_{n}}\sum_{i\le a_{n}}x_i\right| &=\left|\frac{a_{n-1}}{a_{n}}x_{a_{n}}-\frac{1}{a_{n}}\sum_{i\le n-1}x_{a_i}|A_i|\right|   \\
& \le \frac{a_{n-1}}{a_{n}}+\frac{1}{a_{n}}\sum_{i\le n-1}|A_i|=2\,\frac{a_{n-1}}{a_{n}} \to 0,
\end{split}
\end{displaymath}
which completes the proof.
\end{proof}

Lastly, we define the function $g:\mathbf{N}\to \mathbf{N}$ by $g(n):=\lceil \frac{n}{3}\rceil$ for all $n \in \mathbf{N}$ (so that $g(f(n))=k$ if $n \in A_{3k-2}\cup A_{3k-1}\cup A_{3k}$), and set
$$
\mathcal{S}:=\{x \in \{0,1\}^{\mathbf{N}}: x_n \to_{\mathcal{U}} 0\},
$$
where $\mathcal{U}:=\{f^{-1}[g^{-1}[A]]: A\in \mathcal{I}\}$ (note that $\mathcal{U}$ is a maximal ideal on $\mathbf{N}$).

\begin{claim}\label{claim:identity}
$\Lambda(\mathcal{J}\cap \mathcal{F})=\mathcal{S}$.
\end{claim}
\begin{proof}
First, we show that $\mathcal{S}\subseteq \Lambda(\mathcal{J}\cap \mathcal{F})$. Fix a sequence $y \in \mathcal{S}$ and set $U:=\{n\in \mathbf{N}: y_n=1\}$. 
We need to prove that there exists $x \in \mathcal{J} \cap \mathcal{F}$ 
such that $\Lambda(x)=y$. Note that there is exactly one sequence $x \in \mathcal{F}$ such that $\Lambda(x)=y$, that is, the unique one in $\mathcal{F}$ such that $x_{a_{3n-1}}=y_n$ for all $n \in \mathbf{N}$. Let us show that $x \in \mathcal{J}$, i.e., $\frac{1}{n}\sum_{i\le n}x_i\to_{\mathcal{I}}0$. Since $y \in S$, then $U \in \mathcal{U}$, hence 
$$
\textstyle \tilde{U}:=f^{-1}[g^{-1}[U]]=\bigcup_{u \in U} (A_{3u} \cup A_{3u-1}\cup A_{3u-2})\in \mathcal{I}.
$$
This implies, thanks to Claim \ref{claim:estimatefactorials} (see also Figure \ref{fig:graph} below), that, for each $\varepsilon>0$, there exists a finite set $F_\varepsilon \in \mathrm{Fin}$ such that
$$
\textstyle \left\{n \in \mathbf{N}: \frac{1}{n}\sum_{i\le n}x_i\ge \varepsilon\right\} \subseteq F_\varepsilon \cup \tilde{U} \in \mathcal{I}.
$$
Therefore $x \in \mathcal{J} \cap \mathcal{F}$ and $\Lambda(x)=y$. 

Conversely, let us prove that $\Lambda(\mathcal{J}\cap \mathcal{F})\subseteq \mathcal{S}$. Fix a sequence $x \in \mathcal{J} \cap \mathcal{F}$. Then we need to show that $\Lambda(x) \in \mathcal{S}$, that is, $x_{a_{3n-1}} \to_{\mathcal{U}} 0$. This is equivalent to $V:=\{n \in \mathbf{N}: x_{a_{3n-1}}=1\} \in \mathcal{U}$ and also, by the definition of $\mathcal{U}$, to 
$$
\textstyle  \tilde{V}:=f^{-1}\left[g^{-1}\left[V\right]\right]=\bigcup_{v \in V}(A_{3v}\cup A_{3v-1}\cup A_{3v-2}) \in \mathcal{I}.
$$
By assumption we know $\frac{1}{n}\sum_{i\le n}x_i \to_{\mathcal{I}} 0$ and the sequence $x$ is constant on each interval $A_n$, with $x_{a_{3n}}=0$ and $x_{a_{3n-1}}=x_{a_{3n-2}}$ for all $n$. Fix a sufficiently small $\varepsilon>0$, hence 
$
M:=\left\{n \in \mathbf{N}: \frac{1}{n}\sum_{i\le n}x_i\ge \varepsilon\right\} \in \mathcal{I}.
$
Thanks to Claim \ref{claim:estimatefactorials}, we obtain
$$
\textstyle G:=\left(\bigcup_{v \in V}A_{3v-1}\right)\setminus M \in \mathrm{Fin}.
$$
Therefore, thanks to Claim \ref{claim:existencean}, we conclude that 
$$
\textstyle \tilde{V} \subseteq \left(\bigcup_{n \not\equiv 2\bmod{3}}A_n\right)\cup \bigcup_{v \in V}A_{3v-1} \subseteq \left(\bigcup_{n \not\equiv 2\bmod{3}}A_n\right) \cup M \cup G \in \mathcal{I},
$$
which completes the proof.
\end{proof}

Identifying $\{0,1\}$ with the additive group $\mathbf{Z}_2$, we have that $\mathcal{S}$ is a subgroup of the compact group $\{0,1\}^{\mathbf{N}}$. Moreover, $\mathcal{S}$ is not closed (since $\bm{1}_{\{1,\ldots,k\}}(n)\to_{\mathcal{U}}0$ for all $k \in \mathbf{N}$ but $\bm{1}_{\mathbf{N}}(n) \not\to_{\mathcal{U}}0$) and it has finite index (since $\mathcal{U}$ is a maximal ideal, then there are exactly two cosets of $\mathcal{S}$). It follows by \cite[Proposition 1.1(c)]{MR3441133} that $\mathcal{S}$ is nonmeasurable. Thus, thanks to Claim \ref{claim:identity}, $\Lambda(\mathcal{J}\cap \mathcal{F})$ is nonmeasurable.
\end{proof}

%%%%%%%%%%%%%%%%%%%%%%%%%%%%%%%%%%%%%%%%%%%%%%%%%%%%%%%%%%%%%%%%%%%%%%%%%%

\begin{figure}[!htb]
\centering
\begin{tikzpicture}
[scale=1.6]%,description/.style={fill=white, inner sep=2pt}]
\node (xbeg) at (-.4,0){};
\node (xbega) at (.2,0){};
\node (xbegb) at (.7,0){};
\node (xend) at (7.5,0){};
\node (ybeg) at (0,-.4){};
\node (yend) at (0,2.5){};
\draw[thin] (xbeg) -- (xbega);
\draw[dashed] (.1,0) -- (.8,0);
\draw[-latex] (xbegb) -- (xend);
\draw[-latex] (ybeg) -- (yend);

\draw[thin](-.06,2)--(.06,2);
\node (ordinate) at (-.3,2){{\footnotesize $1$}};
\draw [dashdotted, very thin] (.15,2)-- (7.3,2);

\node (a3n3) at (.7,-.25){{\footnotesize $a_{3u-3}$}};
\node (a3n2) at (2,-.25){{\footnotesize $a_{3u-2}$}};
\node (a3n1) at (4,-.25){{\footnotesize $a_{3u-1}$}};
\node (a3n0) at (7,-.25){{\footnotesize $a_{3u}$}};
\draw (1,-.06)--(1,.06);
\draw (2,-.06)--(2,.06);
\draw (4,-.06)--(4,.06);
\draw (7,-.06)--(7,.06);

\draw[thin](-.06,.3)--(.06,.3);
\node (ordinate) at (-.2,.3){{\tiny $\varepsilon$}};
\draw [dashdotted, very thin] (.15,.3)-- (7.3,.3);

\draw[thin](-.06,1.7)--(.06,1.7);
\node (ordinate) at (-.35,1.7){{\tiny $1-\varepsilon$}};
\draw [dashdotted, very thin] (.15,1.7)-- (7.3,1.7);

\node (n3) at (1,.23){{\tiny $\bullet$}};
\node (n2) at (2,2-.15){{\tiny $\bullet$}};
\node (n1) at (4,2-.07){{\tiny $\bullet$}};
\node (n0) at (7,.07){{\tiny $\bullet$}};

%%\draw[very thick, smooth, samples=100, domain=.6:5] plot ({ \x, 3/\x });%NO
%\draw[smooth] plot coordinates{(n3) (n2) (n1)};%NO
\draw[thick] (1,.23) to[out=80, in=185] (2,2-.15) to[out=5, in=180](4,2-.07);
\draw[thick] (4,2-.07) to[out=280, in=170] (7,.07);

%\node (a) at (1.5,.11){{\tiny $A_{3u-2}$}};
%\node (a) at (3,.11){{\tiny $A_{3u-1}$}};
%\node (a) at (5.5,.11){{\tiny $A_{3u}$}};

\draw[dotted] (2,0)-- (2, 2-.15);
\draw[dotted] (4,0)-- (4, 2-.07);

%\draw (1.04,-.09)--(1.04,.09);
\draw[dotted] (1.04,-.5)-- (1.04, .3);
%\draw (5.9,-.09)--(5.9,.09);
\draw[dotted] (5.9,-.5)-- (5.9, .3);

\draw [decorate,decoration={brace,amplitude=10pt}] (5.9,-.5)--(1.04,-.5);% node [black,midway,xshift=9pt] {\footnotesize $P_2$};
\node (jhgf) at (.52+2.95, -.9){{\tiny $\left\{n \in A_{3u-2}\cup A_{3u-1}\cup A_{3u}: \frac{1}{n}\sum_{i\le n}x_i\ge \varepsilon\right\}$}};
\end{tikzpicture}
%\par
%\vspace{2mm}
\caption{Graph of the sequence $\left(\frac{1}{n}\sum_{i\le n}x_i\right)$ in the case $x_{a_{3u-1}}=1$.}
\label{fig:graph}
\end{figure}
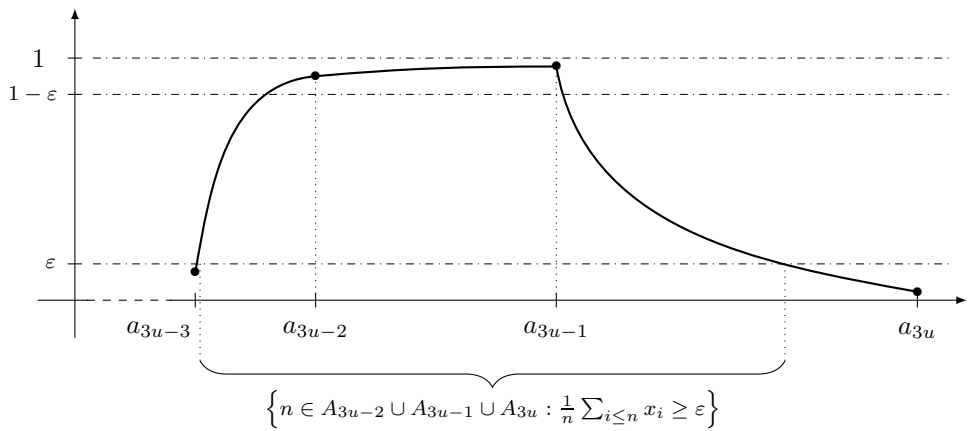

%\clearpage
%\nocite{*}
\bibliographystyle{amsplain}
\bibliography{ideale}

\end{document}